\documentclass[11pt,french,a4paper]{smfart}
\usepackage{amsfonts, amssymb, amsmath, amsthm, latexsym, enumerate, epsfig, color, mathrsfs, fancyhdr, pdfsync, eurosym, endnotes, stmaryrd, accents}
\usepackage[all]{xy}
 
\usepackage[frenchb,english]{babel}

\usepackage[mac]{inputenc}
\usepackage[T1]{fontenc}

\usepackage{hyperref}

\newcommand{\Z}{\ensuremath{\mathbf{Z}}}
\newcommand{\Q}{\ensuremath{\mathbf{Q}}}
\newcommand{\R}{\ensuremath{\mathbf{R}}}

\newcommand{\D}{\ensuremath{\mathbf{D}}}

\renewcommand{\P}{\ensuremath{\mathbf{P}}}

\newcommand\ti[1]{\ensuremath{\tilde{#1}}}
\newcommand{\wti}[1]{\ensuremath{\widetilde{#1}}}

\newcommand{\ho}{\mathbin{\hat{\otimes}}}

\DeclareMathOperator{\red}{red}

\DeclareMathOperator{\Spec}{Spec}
\newcommand{\Ker}{\textrm{Ker}}

\DeclareMathOperator{\Card}{Card}
\DeclareMathOperator{\Int}{Int}
\newcommand{\mon}{\textrm{mon}}

\newcommand{\Xk}{\mathfrak{X}}

\newcommand{\As}{\mathscr{A}}
\newcommand{\Bs}{\mathscr{B}}

\newcommand{\Hs}{\mathscr{H}}

\newcommand{\Os}{\mathscr{O}}

\newcommand{\Mc}{\mathcal{M}}

\newcommand{\Pc}{\mathcal{P}}

\newcommand{\cn}[2]{\ensuremath{\llbracket{#1},{#2}\rrbracket}}
\newcommand{\of}[3]{\ensuremath{\mathopen{#1}{#2}\mathclose{#3}}}

\newcommand{\la}{\ensuremath{\langle}}
\newcommand{\ra}{\ensuremath{\rangle}}

\newcommand{\eps}{\ensuremath{\varepsilon}}

\newcommand{\br}{\ensuremath{{\boldsymbol{r}}}}

\newcommand\wc{{\mkern 2mu\cdot\mkern 2mu}}
\newcommand\va{|\wc|}

\newcommand\cf{\textit{cf}.}

\theoremstyle{plain}
\newtheorem{thm}{Th\'eor\`eme}[section]
\newtheorem{cor}[thm]{Corollaire}
\newtheorem{prop}[thm]{Proposition}
\newtheorem{lem}[thm]{Lemme}

\newtheorem{defi}[thm]{D\'efinition}

\newtheorem{thmintro}{Th\'eor\`eme}

\theoremstyle{remark}
\newtheorem{rem}[thm]{Remarque}

\begin{document}
\setlength{\baselineskip}{0.52cm}	

\title[Composantes connexes $p$-adiques]{Sur les composantes connexes d'une famille d'espaces analytiques $p$-adiques}
\author{J\'er\^ome Poineau}
\address{Institut de recherche math\'ematique avanc\'ee, 7, rue Ren\'e Descartes, 67084 Strasbourg, France}
\email{jerome.poineau@math.unistra.fr}
\thanks{L'auteur est membre du projet ANR \og GLOBES \fg~: ANR-12-JS01-0007-01.}

\date{\today}


\begin{abstract}
\selectlanguage{french}
Soit~$X=\Mc(\As)$ un espace affino\"{\i}de et soient $f,g\in\As$. Nous \'etudions les ensembles de composantes connexes des espaces d\'efinis par une in\'egalit\'e de la forme $|f|\le r\, |g|$, avec $r\ge 0$. Nous montrons qu'il existe une partition finie de~$\R_{+}$ en intervalles sur lesquels ces ensembles sont canoniquement en bijection et que les bornes de ces intervalles appartiennent \`a~$\sqrt{\rho(\As)}$.

\vskip 0.5\baselineskip

\selectlanguage{english}
\noindent{\bf Abstract}
\vskip 0.5\baselineskip
\noindent
{\bf On the connected components of a family of $p$-adic analytic spaces.} Let~$X=\Mc(\As)$ be an affinoid space and let $f,g\in\As$. We study the sets of connected components of the spaces defined by an inequality of the form $|f|\le r\, |g|$, with $r\ge 0$. We prove that there exists a finite partition of~$\R_{+}$ into intervals where those sets are canonically in bijection and that the bounds of those intervals belong to~$\sqrt{\rho(\As)}$.
\end{abstract}

\maketitle

Dans ce texte, nous nous int\'eressons \`a la variation des composantes connexes d'une famille particuli\`ere d'espaces analytiques $p$-adiques ou, plus g\'en\'eralement, d'espaces analytiques d\'efinis sur un corps ultram\'etrique. Pour aborder un tel probl\`eme, il est commode de se placer dans le cadre des espaces analytiques d\'efinis par V.~Berkovich (\cf~\cite{rouge, bleu}). Sous-jacent \`a ceux-ci se trouve en effet un v\'eritable espace topologique (par opposition \`a un site) qui jouit, de surcro\^{\i}t, de bonnes propri\'et\'es~: compacit\'e locale ou connexit\'e par arcs locale, pour ne citer que les plus simples.

La topologie des espaces de Berkovich a r\'ecemment fait l'objet de plusieurs \'etudes approfondies qui ont r\'ev\'el\'e son caract\`ere mod\'er\'e. Dans le cas lisse, on sait, par exemple, depuis~\cite{smoothI}, que ces espaces ont le type d'homotopie de CW-complexes. Ce r\'esultat a \'et\'e \'etendu aux analytifi\'ees de vari\'et\'es quasi-projectives par E.~Hrushovski et F.~Loeser, \textit{via} l'utilisation de th\'eorie des mod\`eles. L'introduction de ces nouvelles techniques a \'egalement permis de d\'emontrer de nombreux r\'esultats de m\^eme nature que ceux dont on dispose en g\'eom\'etrie alg\'ebrique r\'eelle ou o-minimale. Nous renvoyons le lecteur int\'eress\'e \`a l'article original~\cite{HL} ou \`a sa recension~\cite{moderes} pour le s\'eminaire Bourbaki.

Notre article s'inscrit dans la lign\'ee de ces travaux. Afin de donner un \'enonc\'e pr\'ecis du th\'eor\`eme principal que nous souhaitons d\'emontrer, fixons, et ce pour toute la suite du texte, un corps~$k$ muni d'une valeur absolue ultram\'etrique~$\va$ pour laquelle il est complet.

\begin{thmintro}\label{thm:partition}
Soit~$X=\Mc(\As)$ un espace $k$-affino\"{\i}de. Soient $f ,g\in \As$. Notons~$\sqrt{\rho(\As)}$ le sous-$\Q$-espace vectoriel de~$\R_{+}^\ast$ engendr\'e par les valeurs non nulles de la norme spectrale sur~$\As$. Il existe une partition finie~$\Pc$ de~$\R_{+}$ de la forme
\[\Pc = \{\of{[}{0,a_{0}}{[}, \of{[}{a_{0},a_{1}}{[},\dotsc, \of{[}{a_{n-1},a_{n}}{[}, \of{[}{a_{n},+\infty}{[}\},\]
avec $a_{0},\dotsc,a_{n} \in \sqrt{\rho(\As)}$, satisfaisant la condition suivante~: quel que soit $I\in \Pc$ et quels que soient $r\le s \in I$, l'application naturelle
\[\pi_{0}(\{x\in X \mid |f(x)|\le r\, |g(x)|\}) \to \pi_{0}(\{x\in X \mid |f(x)|\le s\, |g(x)|\})\]
est bijective.

Le m\^eme r\'esultat vaut en rempla\c{c}ant l'ensemble des composantes connexes par l'ensemble des composantes connexes g\'eom\'etriques, l'ensemble des composantes irr\'eductibles ou l'ensemble des composantes irr\'eductibles g\'eom\'etriques.
\end{thmintro}

Mentionnons que, lorsque~$X$ est strictement $k$-affino\"{\i}de, ce qui correspond au cas rigide classique, on a $\sqrt{\rho(\As)} = \sqrt{|k^\ast|}$.

\medbreak

En d\'epit de son apparence purement topologique, ce r\'esultat poss\`ede des cons\'equences arithm\'etiques profondes. A.~Abbes et T.~Saito l'\'enoncent, pour des ensembles d\'efinis par des \'equations de la forme $|f|\ge \eps$ et dans une version affaiblie, dans l'article~\cite{Abbes} et l'utilisent pour d\'efinir la filtration de ramification d'un corps local \`a corps r\'esiduel imparfait. Pr\'ecisons que les informations sur la nature des bornes fournissent, dans ce contexte, un analogue du th\'eor\`eme de Hasse-Arf.

Indiquons encore que nous avons propos\'e une premi\`ere d\'emonstration de ce th\'eor\`eme, toujours pour des ensembles d\'efinis par des \'equations de la forme $|f|\ge \eps$, dans~\cite{Compositio}. Dans la m\'ethode utilis\'ee alors, on interpr\`ete les espaces \'etudi\'es comme les fibres d'un morphisme analytique. On peut alors choisir un mod\`ele formel de ce morphisme et \'etudier sa fibre sp\'eciale, un morphisme de sch\'emas, par les m\'ethodes classiques de g\'eom\'etrie alg\'ebrique expos\'ees en d\'etails dans les EGA. Revenir du morphisme sp\'ecialis\'e au morphisme g\'en\'erique est une \'etape d\'elicate qui n'est possible que sous certaines hypoth\`eses sur le mod\`ele formel choisi. Pour assurer l'existence d'un mod\`ele convenable, nous avons d\^u invoquer des th\'eor\`emes difficiles, sortes de versions faibles de la r\'esolution des singularit\'es (th\'eor\`eme de la fibre r\'eduite~\cite{FIV} et th\'eor\`eme d'\'elimination de la ramification sauvage~\cite{Epp}).

Mentionnons \'egalement que, dans le cas o\`u toutes les donn\'ees de l'\'enonc\'e sont alg\'ebriques ($X$ est d\'efini par des in\'egalit\'es entre polyn\^omes dans l'analytification d'une vari\'et\'e quasi-projective, etc.), alors le r\'esultat d\'ecoule de~\cite[theorem~13.4.3]{HL}, et m\^eme dans une version bien plus forte requ\'erant que les inclusions induisent des \'equivalences homotopiques.

Nous proposons ici une preuve simple du th\'eor\`eme~\ref{thm:partition} se dispensant notamment de tout recours aux sch\'emas formels. Nous utilisons des techniques classiques en g\'eom\'etrie de Berkovich --- la notion de bord de Shilov appara\^{\i}t de fa\c{c}on cruciale --- m\^el\'ees \`a des arguments de r\'eduction au sens de la norme spectrale, dans la version raffin\'ee gradu\'ee propos\'ee par M.~Temkin dans~\cite{TemkinII}.

\medbreak

Concluons cette introduction par quelques mots sur la structure de l'article. Dans la premi\`ere section, on d\'emontre une partie cons\'equente du th\'eor\`eme~\ref{thm:partition}~: l'existence d'une partition finie en intervalles de la forme voulue et v\'erifiant les propri\'et\'es d\'esir\'ees, mais sans information sur la nature des bornes des intervalles. La preuve repose sur une observation assurant que, dans un espace affino\"{\i}de~$X$, les composantes connexes d'un domaine affino\"{\i}de d\'efini par une in\'egalit\'e du type $|f|\le r\, |g|$, avec~$f$ et~$g$ sans z\'eros communs, peuvent \^etre rep\'er\'ees par des points du bord de Shilov de~$X$. L'observation elle-m\^eme se d\'emontre \`a l'aide d'un analogue du principe du maximum.

Dans la seconde section, nous consid\'erons un espace affino\"{\i}de~$X$ et l'application de r\'eduction associ\'ee $\red\colon X \to \ti X$. Dans~\cite{Bosch}, S.~Bosch d\'emontre que, sous certaines hypoth\`eses, l'image inverse d'un point ferm\'e de la $\ti k$-vari\'et\'e~$\ti X$ est connexe. Nous montrons que ce r\'esultat vaut sans autre hypoth\`ese que l'\'equidimensionalit\'e de~$X$ et l'\'etendons aux ferm\'es de Zariski connexes de~$\ti X$.

La troisi\`eme section contient des rappels sur le lieu $H$-strict d'un espace analytique, $H$ \'etant un sous-groupe de~$\R_{+}^\ast$. Grossi\`erement parlant, ce lieu est constitu\'e des points qui poss\`edent un voisinage qui pouvant \^etre d\'efini par des in\'egalit\'es o\`u n'interviennent que des nombres r\'eels appartenant \`a~$H$. D'apr\`es des r\'esultats de~\cite{TemkinII} et~\cite{CTdescent}, l'appartenance d'un point au lieu $H$-strict peut se lire sur la r\'eduction du germe de l'espace en ce point.

Dans la derni\`ere section, nous appliquons ces techniques pour d\'emontrer que, si $X$ est un espace analytique $H$-strict, $f$ et~$g$ des fonctions analytiques sur~$X$ sans z\'eros communs et~$r$ un \'el\'ement de $\R_{+}^\ast \setminus \sqrt{H}$, les ensembles de composantes connexes des lieux d\'efinis respectivement par les in\'egalit\'es $|f| < r\, |g|$ et $|f|\le r\, |g|$ sont canoniquement en bijection. Ce r\'esultat et ses g\'en\'eralisations nous semblent pr\'esenter un int\'er\^et ind\'ependant, de m\^eme que les techniques, bas\'ees sur la r\'eduction des germes, mises en {\oe}uvre pour les d\'emontrer. 

En regroupant les r\'esultats obtenus, il est alors ais\'e de compl\'eter la preuve du th\'eor\`eme~\ref{thm:partition}. Supposons, pour simplifier, que~$f$ et~$g$ ne s'annulent pas simultan\'ement sur~$X$. Choisissons un nombre r\'eel~$r>0$. D'apr\`es la section~1, pour tout $s>r$ assez proche de~$r$, les ensembles $\pi_{0}(\{|f|\le s\, |g|\})$ sont canoniquement en bijection, et donc \'egalement canoniquement en bijection avec $\pi_{0}(\{|f|< r\,|g|\})$. L'espace affino\"{\i}de~$X$ est $H$-strict avec $H=\rho(\As)$. D'apr\`es la section~3, si~$r$ se trouve hors de~$\sqrt{H}$, les ensembles de composantes connexes pr\'ec\'edents s'identifient alors encore \`a $\pi_{0}(\{|f|\le r\, |g|\})$, emp\^echant ainsi le nombre r\'eel~$r$ d'\^etre une borne de l'un des intervalles consid\'er\'es dans le th\'eor\`eme. 

Pour terminer, nous voudrions souligner que, comme il appara\^{\i}t dans l'argument pr\'ec\'edent, m\^eme si l'on ne souhaite d\'emontrer le r\'esultat que pour les espaces strictement affino\"{\i}des de la g\'eom\'etrie rigide classique, la m\'ethode pr\'esent\'ee ici impose de consid\'erer des domaines affino\"{\i}des d\'efinis par une in\'egalit\'e du type $|f|\le r\, |g|$, o\`u~$r$ est arbitraire, et donc de sortir du cadre strictement affino\"{\i}de. 

\section{Variation de connexit\'e}

Soit~$X=\Mc(\As)$ un espace $k$-affino\"{\i}de. Nous noterons~$\Gamma(X)$ son bord de Shilov (\cf~\cite[section~2.4]{rouge}).

Commen\c{c}ons par un r\'esultat \'el\'ementaire.

\begin{lem}\label{lem:injectif}
Soient $f,g \in \As$ et $r\ge 0$. Il existe $\eps>0$ tel que, quel que soit $s \in \of{[}{r, r+\eps}{]}$, l'application naturelle
\[\pi_{0}(\{x\in X \mid |f(x)|\le r\, |g(x)|\}) \to \pi_{0}(\{x\in X \mid |f(x)|\le s\, |g(x)|\})\]
soit injective.
\end{lem}
\begin{proof}
Nous pouvons supposer que $V = \{x\in X \mid |f(x)|\le r\, |g(x)|\}$ n'est pas vide. Consid\'erons ses composantes connexes $C_{1},\dotsc,C_{n}$. Ce sont des parties compactes de~$X$. Puisque~$X$ est s\'epar\'e, nous pouvons trouver des ouverts $U_{1},\dotsc,U_{n}$ de~$X$ deux \`a deux disjoints tels que pour tout~$i$, $U_{i}$ contient~$C_{i}$. Notons~$U$ la r\'eunion des~$U_{i}$. L'application naturelle $\pi_{0}(V) \to \pi_{0}(U)$ est alors injective, et le m\^eme r\'esultat vaut en rempla\c{c}ant~$U$ par n'importe laquelle de ses parties qui contient~$V$. 

Si~$U = X$, le r\'esultat vaut pour tout choix de~$\eps>0$. Supposons d\'esormais que $U \ne X$. Dans ce cas, la partie~$X\setminus U$ est compacte et non vide. Puisqu'elle est disjointe de~$V$, la fonction~$|g|/|f|$ y est donc d\'efinie et y poss\`ede un maximum~$M < r^{-1}$. Choisissons un nombre r\'eel~$\eps>0$ tel que $r+\eps < M^{-1}$. Quel que soit $s\in\of{[}{r,r+\eps}{]}$, la partie $\{x\in X \mid |f(x)|\le s\, |g(x)|\}$ est alors contenue dans~$U$, ce qui permet de conclure.
\end{proof}

\begin{prop}\label{prop:ccbord}
Soient $f,g \in \As$ tels que $(f,g) = (1)$. Soit $r>0$. Notons~$V(f)$ le ferm\'e de Zariski de~$X$ d\'efini par~$f$. Si $\{x\in X \mid |f(x)|\le r \, |g(x)|\}$ n'est pas vide, alors chacune de ses composantes connexes contient une composante connexe de~$V(f)$ ou un point du bord de Shilov de~$X$ en lequel ni~$f$ ni~$g$ ne s'annule.
\end{prop}
\begin{proof}
Supposons que $\{x\in X \mid |f(x)|\le r\, |g(x)|\}$ n'est pas vide. Soit~$C$ l'une de ses composantes connexes. Si elle coupe~$V(f)$, alors elle contient l'une de ses composantes connexes.

Supposons donc que~$f$ ne s'annule pas sur~$C$. D'apr\`es le lemme pr\'ec\'edent, il existe $s > r$ et une composante connexe~$D$ de $\{x\in X \mid |f(x)|\le s\,|g(x)|\}$ telle que
\[C = D \cap \{x\in X \mid |f(x)|\le r\,|g(x)|\}.\]
Nous pouvons en outre choisir~$s$ de sorte que~$f$ ne s'annule pas sur~$D$. La partie~$D$ est alors un domaine affino\"{\i}de de~$X$ sur laquelle $g/f$ d\'efinit une fonction analytique~$h$. Consid\'erons un point~$\gamma$ du bord de Shilov de~$D$ en lequel la valeur absolue de~$h$ est maximale. Nous avons alors $|h(\gamma)|\ge r^{-1}$ et le point~$\gamma$ appartient \`a~$C$, donc \`a l'int\'erieur de~$D$. D'apr\`es le principe du maximum~\cite[proposition~2.5.20]{rouge}, c'est un point du bord de Shilov de~$X$.
\end{proof}

\begin{cor}\label{cor:ccbordavecannulation}
Soient $f,g \in \As$ et $r\ge 0$. Notons~$V(f)$ le ferm\'e de Zariski de~$X$ d\'efini par~$f$. Si $\{x\in X \mid |f(x)|\le r \, |g(x)|\}$ n'est pas vide, alors chacune de ses composantes connexes contient une composante connexe de~$V(f)$ ou un point du bord de Shilov de~$X$ en lequel ni~$f$ ni~$g$ ne s'annule.
\end{cor}
\begin{proof}
Supposons que $\{x\in X \mid |f(x)|\le r\, |g(x)|\}$ n'est pas vide. Soit~$C$ l'une de ses composantes connexes et supposons qu'elle ne contient aucun point du bord de Shilov de~$X$ en lequel ni~$f$ ni~$g$ ne s'annule.

Nous allons montrer que la fonction~$g$ s'annule sur~$C$. En effet, sinon, pour tout~$\eps>0$ assez petit, la proposition pr\'ec\'edente appliqu\'ee \`a l'espace $X_{\eps} = \{x\in X \mid |g(x)|\ge \eps\}$ assure que~$C$ contient un point~$y$ du bord de Shilov de~$X_{\eps}$ en lequel ni~$f$ ni~$g$ ne s'annule. Par hypoth\`ese, ce point ne peut appartenir au bord de Shilov de~$X$. D'apr\`es le principe du maximum~\cite[proposition~2.5.20]{rouge}, il appartient donc au bord de~$X_{\eps}$ dans~$X$ et nous avons donc $|g(y)|=\eps$. En utilisant la compacit\'e de~$C$, on montre alors que~$g$ s'annule sur~$C$.

On en d\'eduit que~$f$ s'annule \'egalement sur~$C$ et donc que~$C$ contient une composante connexe de~$V(f)$.
\end{proof}

\begin{cor}\label{cor:surjectifintervalle}
Soient $f,g \in \As$. Posons
\[|f/g(\Gamma(X))^\ast| = \{|f(x)|/|g(x)| \mid x\in\Gamma(X), f(x)\ne 0, g(x)\ne 0\}.\]
C'est un sous-ensemble fini de~$\R_{+}^\ast$. Soient~$a < b \in \R_{+}$ tels que $\of{]}{a,b}{[} \cap |f/g(\Gamma(X))^\ast| = \emptyset$. Alors, pour tous $r\le s \in \of{[}{a,b}{[}$, l'application naturelle
\[\pi_{0}(\{x\in X \mid |f(x)|\le r\,|g(x)|\}) \to \pi_{0}(\{x\in X \mid |f(x)|\le s\,|g(x)|\})\]
est surjective. En particulier, l'application
\[r \in \of{[}{a,b}{[} \mapsto \Card(\pi_{0}(\{x\in X \mid |f(x)|\le r\, |g(x)| \}))\]
est d\'ecroissante.
\end{cor}

On peut d\'ecliner ce corollaire sous diff\'erentes formes, en faisant, par exemple, intervenir des sch\'emas formels. Nous ignorons si le r\'esultat suivant est connu.

\begin{cor}
Soit~$\Xk$ un $k^\circ$-sch\'ema formel affine admissible dont la fibre sp\'eciale~$\Xk_{s}$ est r\'eduite. Notons~$\Xk_{\eta}$ sa fibre g\'en\'erique. Soit~$g$ un \'el\'ement de~$\Os(\Xk)$ qui ne s'annule sur aucune composante irr\'eductible de~$\Xk_{s}$. Alors, pour tous $s\le r \in \of{[}{0,1}{]}$, l'application naturelle
\[\pi_{0}(\{x\in \Xk_{\eta} \mid |g(x)|\ge r\}) \to \pi_{0}(\{x\in \Xk_{\eta} \mid |g(x)|\ge s\})\]
est surjective.
\end{cor}
\begin{proof}
Sous les hypoth\`eses de l'\'enonc\'e, la fibre g\'en\'erique~$\Xk_{\eta}$ est un espace affino\"{\i}de et l'application de r\'eduction $\Xk_{\eta} \to \Xk_{s}$ co\"{\i}ncide avec celle induite par la norme spectrale sur~$\Xk_{\eta}$ (\cf~\cite[proposition~1.1]{FIV}). En outre, en tout point du bord de Shilov de~$\Xk_{\eta}$, on a $|g(\gamma)|=1$. On peut alors appliquer le corollaire qui pr\'ec\`ede.
\end{proof}

Nous pouvons \'egalement d\'ej\`a d\'emontrer une partie du th\'eor\`eme~\ref{thm:partition} pr\'esent\'e en introduction.

\begin{cor}\label{cor:partitionsansrhoAs}
Soient $f,g \in \As$. Il existe une partition finie~$\Pc$ de~$\R_{+}$ de la forme
\[\Pc = \{\of{[}{0,a_{0}}{[}, \of{[}{a_{0},a_{1}}{[},\dotsc, \of{[}{a_{p-1},a_{p}}{[}, \of{[}{a_{p},+\infty}{[}\}\] 
satisfaisant la condition suivante~: quel que soit $I\in \Pc$ et quels que soient $r\le s \in I$, l'application naturelle
\[\pi_{0}(\{x\in X \mid |f(x)|\le r\, |g(x)|\}) \to \pi_{0}(\{x\in X \mid |f(x)|\le s\, |g(x)|\})\]
est bijective.

Le m\^eme r\'esultat vaut en rempla\c{c}ant l'ensemble des composantes par l'ensemble des composantes connexes g\'eom\'etriques, l'ensemble des composantes irr\'eductibles ou l'ensemble des composantes irr\'eductibles g\'eom\'etriques.
\end{cor}
\begin{proof}
D'apr\`es le corollaire~\ref{cor:surjectifintervalle}, il existe une partition de la forme voulue sur les \'el\'ements de laquelle les applications entre~$\pi_{0}$ sont surjectives. Puisque le bord de Shilov de~$X$ est finie, cette partition est finie.

On raffine cette partition en une partition satisfaisant la propri\'et\'e requise \`a l'aide du lemme~\ref{lem:injectif}.

La derni\`ere partie du r\'esultat se d\'emontre en remplacant~$X$ par~$X\ho_{k} k^a$, o\`u~$k^a$ d\'esigne le compl\'et\'e d'une cl\^oture alg\'ebrique de~$k$, sa normalis\'ee~$\ti{X}$ ou $\ti{X}\ho_{k} k^a$.
\end{proof}

Nous tirons maintenant des cons\'equences de ce r\'esultat. Nous nous contenterons ici d'esquisser les preuves en renvoyant \`a~\cite{Compositio} pour de plus amples d\'etails.

\begin{cor}\label{cor:ggeepsirreductible}
Soit $g \in \As$. Si~$X$ est irr\'eductible, alors il existe~$r_{0}>0$ tel que, pour tout $r\in \of{]}{0,r_{0}}{]}$, le domaine affino\"{\i}de $\{x\in X \mid |g(x)|\ge r\}$ soit irr\'eductible.
\end{cor}
\begin{proof}
On d\'eduit du corollaire~\ref{cor:partitionsansrhoAs} que, pour~$r>0$ assez proche de~0, l'inclusion
\[\{x\in X \mid |g(x)|\ge r\} \subset \{x\in X \mid g(x) \ne 0\}\]
induit une bijection entre les ensembles de composantes irr\'eductibles. Or, d'apr\`es~\cite{Bart} et~\cite{Lu}, l'ensemble de droite est irr\'eductible.
\end{proof}

\begin{cor}
Un point d'un bon espace $k$-analytique en lequel l'anneau local est int\`egre poss\`ede une base de voisinages affino\"{\i}des irr\'eductibles.
\end{cor}
\begin{proof}
On se ram\`ene imm\'ediatement \`a montrer qu'un point~$x$ situ\'e sur une seule composante irr\'eductible~$Z$ d'un espace $k$-affino\"{\i}de~$X$ poss\`ede un voisinage affino\"{\i}de irr\'eductible. Consid\'erons une fonction $g\in \Os(X)$ qui ne s'annule pas en~$x$, mais s'annule identiquement sur toute composante irr\'eductible de~$X$ diff\'erente de~$Z$. Le corollaire pr\'ec\'edent assure alors que le domaine affino\"{\i}de $\{x\in X \mid g(x) \ge \eps\}$ est irr\'eductible pour~$\eps>0$ assez petit, ce qui permet de conclure.
\end{proof}

\section{Connexit\'e des tubes}

Fixons un espace $k$-affino\"{\i}de $X=\Mc(\As)$. Notons~$\ti{X}$ sa r\'eduction et $\red \colon X \to \ti{X}$ l'application de r\'eduction.

Un r\'esultat classique de S.~Bosch permet d'obtenir des parties connexes de~$X$ en utilisant cette application.

\begin{thm}[S.~Bosch, \protect{\cite[Satz~6.1]{Bosch}}]\label{thm:Boschoriginal}
Supposons que la valuation de~$k$ n'est pas triviale et que~$X$ est strictement $k$-affino\"{\i}de, distingu\'e et \'equidimensionnel. Alors, pour tout point ferm\'e~$\ti x$ de~$\ti X$, le tube $\red^{-1}(\ti x)$ est connexe.
\end{thm}

Nous souhaitons g\'en\'eraliser ce r\'esultat en conservant uniquement l'hypoth\`ese d'\'equidimensionalit\'e. 

Nous renvoyons \`a~\cite[chapter~9]{rouge} ou~\cite{variationdimension} pour la notion de dimension $k$-analytique d'un espace $k$-affino\"{\i}de (ou plus g\'en\'eralement $k$-analytique, m\^eme si nous n'en aurons pas l'utilit\'e ici). Dans le cas strictement $k$-affino\"{\i}de, elle co\"{\i}ncide avec la dimension de Krull. Nous dirons que l'espace $k$-affino\"{\i}de~$X$ est \'equidimensionnel si toutes ses composantes irr\'eductibles ont la m\^eme dimension $k$-analytique. Dans ce cas, pour toute extension valu\'ee compl\`ete~$L$ de~$k$, l'espace $L$-affino\"{\i}de $X \ho_{k} L$ est encore \'equidimensionnel.

\begin{lem}
Soit~$K$ une extension finie de~$k$ munie de l'unique valeur absolue qui prolonge celle de~$k$. Soit~$L$ une extension valu\'ee compl\`ete et alg\'ebriquement close du corps~$k$. Alors tout morphisme $\ti{k}$-lin\'eaire $\ti{K} \to \ti{L}$ provient d'un morphisme $k$-lin\'eaire isom\'etrique $K \to L$.
\end{lem}
\begin{proof}
Soit $\psi : \ti{K} \to \ti{L}$ un morphisme $\ti{k}$-lin\'eaire. Supposons tout d'abord que l'extension $k \to K$ n'est pas ramifi\'ee. Dans ce cas, l'extension r\'esiduelle $\ti{k} \to \ti{K}$ est finie et s\'eparable,  donc engendr\'ee par un \'el\'ement~$\ti{\alpha}$. Relevons-le en un \'el\'ement~$\alpha$ de~$K^\circ$. On se convainc ais\'ement que $k[\alpha]=K$. 

Soit $P \in k^\circ[T]$ le polyn\^ome minimal unitaire de~$\alpha$ sur~$k$. Il existe une racine~$\beta$ de~$P$ dans~$L$ dont la r\'eduction est~$\psi(\ti{\alpha})$. On v\'erifie alors que le morphisme $k$-lin\'eaire $\varphi : K \to L$ d\'efini en envoyant~$\alpha$ sur~$\beta$ a pour r\'eduction~$\psi$.

Revenons, \`a pr\'esent, au cas d'une extension $k \to K$ quelconque. Le raisonnement pr\'ec\'edent permet de remplacer~$k$ par sa plus grande extension non ramifi\'ee contenue dans~$K$ et donc de supposer que l'extension $\ti{k} \to \ti{K}$ est radicielle. Notons~$p$ l'exposant caract\'eristique de~$\ti{k}$. Choisissons alors un morphisme $k$-lin\'eaire isom\'etrique $\varphi : K \to L$ quelconque. Son existence est assur\'ee par le fait que~$L$ est alg\'ebriquement clos.

Soit $\ti{\alpha} \in \ti{K}$. Par hypoth\`ese, il existe un entier positif~$n$ et un \'el\'ement~$\ti{\beta}$ de~$\ti{k}$ tels que $\ti{\alpha}^{p^n} = \ti{\beta}$. Nous avons donc \'egalement $\ti{\varphi}(\ti{\alpha})^{p^n} = \ti{\beta}$ et donc $\ti{\varphi}(\ti{\alpha})=\psi(\ti{\alpha})$, puisque le polyn\^ome $X^{p^n} - \ti{\beta}$ poss\`ede une seule racine dans~$\ti{L}$.
\end{proof}

\begin{lem}\label{lem:fibre}
Supposons que l'espace~$X$ est strictement $k$-affino\"ide. Soit~$L$ une extension valu\'ee compl\`ete et alg\'ebriquement close du corps~$k$. Notons $\pi \colon X\ho_{k} L \to X$ le morphisme de changement de base et~$\ti\pi$ le morphisme induit entre les r\'eductions. Soit~$x$ un point rigide de~$X$. Alors, l'application
\[\pi^{-1}(x) \to \ti\pi^{-1}(\ti x)\]
induite par l'application de r\'eduction est surjective.  
\end{lem}
\begin{proof}
Nous pouvons \'ecrire l'application $\pi^{-1}(x) \to \ti\pi^{-1}(\ti x)$ comme la compos\'ee des applications
\[\alpha \colon \Mc\left(\Hs(x)\hat{\otimes}_{k} L\right) \to \Spec\left(\wti{\Hs(x)\hat{\otimes}_{k} L}\right),\]
\[\beta \colon \Spec\left(\wti{\Hs(x)\hat{\otimes}_{k} L}\right) \to \Spec\left(\wti{\Hs(x)}\otimes_{\ti{k}} \ti{L}\right)\]
et
\[\gamma \colon \Spec\left(\wti{\Hs(x)} \otimes_{\ti{k}} \ti{L}\right) \to \Spec\left(\kappa(\ti{x}) \otimes_{\ti{k}} \ti{L}\right).\]
Il suffit de montrer que ces trois applications sont surjectives. Pour~$\alpha$, qui est une application de r\'eduction, ce r\'esultat provient de~\cite[proposition~2.4.4~(i)]{rouge}. Pour~$\beta$, on le d\'eduit du lemme pr\'ec\'edent, car~$\Hs(x)$ est une extension finie de~$k$, et c'est bien connu pour~$\gamma$.
\end{proof}

\begin{lem}\label{lem:staff}
Supposons que la valuation du corps~$k$ n'est pas triviale et que l'espace~$X$ est strictement $k$-affino\"{\i}de et \'equi\-di\-men\-sion\-nel. Soit~$\ti{x}$ un point ferm\'e de~$\ti{X}$. Alors l'ouvert $\red^{-1}(\ti{x})$ est connexe. 
\end{lem}
\begin{proof}
Soient~$x$ et~$y$ deux points de $\red^{-1}(\tilde{x})$. Soit~$L$ une extension valu\'ee compl\`ete et alg\'ebriquement close du corps~$\Hs(y)$ et consid\'erons le diagramme commutatif
\[\xymatrix{
X \ho_{k} L \ar[r]^\pi \ar[d]^\red & X \ar[d]^\red\\
\wti{X\ho_{k} L} \ar[r]^{\ti\pi} & \ti X
}.\]

Remarquons que, pour tout point $\ti L$-rationnel~$\alpha$ de~$\Spec(\wti{X\ho_{k} L})$, le tube~$\red^{-1}(\alpha)$ est connexe. En effet, nous pouvons supposer $X \ho_{k} L$ est r\'eduit, car cela ne change ni l'espace topologique sous-jacent, ni l'application de r\'eduction. Dans ce cas, puisque~$L$ est alg\'ebriquement clos et de valuation non triviale, l'espace strictement $L$-affino\"{\i}de~$X \ho_{k} L$ est distingu\'e, d'apr\`es~\cite[theorem~6.4.3/1]{BGR}. Le r\'esultat d\'ecoule alors du th\'eor\`eme~\ref{thm:Boschoriginal} de S.~Bosch.

Par choix de~$L$, il existe un point $L$-rationnel~$u$ de~$X \ho_{k} L$ au-dessus de~$y$. Sa r\'eduction~$\red(u)$ appartient alors \`a $\ti\pi^{-1}(\ti x)$. D'apr\`es la remarque pr\'ec\'edente, le tube $C=\red^{-1}(\red(u))$ est connexe. D'apr\`es le lemme~\ref{lem:fibre}, il rencontre~$\pi^{-1}(x)$. Par cons\'equent, son image~$\pi(C)$ est une partie connexe de $\red^{-1}(\ti x)$ qui contient \`a la fois les points~$x$ et~$y$.
\end{proof}

Pour g\'en\'eraliser le r\'esultat hors du cadre strictement affino\"{\i}de, nous allons utiliser les r\'eductions gradu\'ees au sens de M.~Temkin (\cf~\cite{TemkinII}). Ce seront les seules que nous utiliserons dans la suite de cette section. Rappelons-en la construction en quelques mots. Soit~$\Bs$ une $k$-alg\`ebre de Banach et d\'esignons par~$\rho$ sa norme spectrale. Pour tout nombre r\'eel~$r>0$, on pose $\Bs^\circ_{r} = \{f\in \Bs \mid \rho(f)\le r\}$, $\Bs^{\circ\circ}_{r} = \{f\in \Bs \mid \rho(f)< r\}$ et $\ti{\Bs}_{r} = \Bs^\circ_{r}/\Bs^{\circ\circ}_{r}$. La r\'eduction gradu\'ee de~$\Bs$ est alors l'alg\`ebre $\R_{+}^\ast$-gradu\'ee
\[\ti\Bs = \bigoplus_{r>0} \ti{\Bs}_{r}.\]
Lorsque~$\Bs$ est une alg\`ebre $k$-affino\"{\i}de, le spectre gradu\'e de~$\ti\Bs$ poss\`ede des propri\'et\'es analogues \`a la r\'eduction dans le cas strictement affino\"{\i}de classique.

Commen\c{c}ons par passer du cas d'un point ferm\'e \`a celui d'un ferm\'e de Zariski connexe quelconque. Notre preuve suit celle de~\cite[lemme~3.1.2]{salg}. Pour la d\'efinition de l'int\'erieur d'un morphisme entre espaces $k$-analytiques, nous renvoyons \`a~\cite[sections~2.5 et~3.1]{rouge}. 

\begin{lem}\label{lem:reductioninterieur}
Soient $\varphi \colon Y \to Z$ un morphisme entre un bon espace $k$-analytique~$Y$ et un espace $k$-affino\"{\i}de~$Z$. Notons $\red \colon Y \to \ti Y$ l'application de r\'eduction gradu\'ee. Alors $\red(\varphi(Y))$ contient l'adh\'erence de Zariski de~$\red(\varphi(\Int(Y/Z)))$. 
\end{lem}
\begin{proof}
Soit $y\in\Int(Y/Z)$. Nous allons montrer que $\red(\varphi(Y))$ contient l'adh\'erence Zariski de~$\wti{\varphi(y)}$.

Par d\'efinition, il existe un voisinage affino\"{\i}de~$Y'$ de~$y$ et un voisinage affino\"{\i}de~$Z'$ de~$\varphi(y)$ tels que $\varphi(Y')\subset Z'$ et $y\in \Int(Y'/Z')$. D'apr\`es~\cite[proposition~3.1.3~(ii) et corollary~2.5.13~(ii)]{rouge}, nous avons $y\in\Int(Y'/X)$. Quitte \`a remplacer~$Y$ par~$Y'$, nous pouvons supposer que~$Y$ est affino\"{\i}de.

Consid\'erons le diagramme commutatif
\[\xymatrix{
Y \ar[r]^\varphi \ar[d]^\red & Z \ar[d]^\red\\
\ti Y \ar[r]^{\ti\varphi} & \ti Z
}.\]

Pour tout polyrayon~$\br$, notons~$\D^+_{\br}$ (resp.~$\D^-_{\br}$) le disque ferm\'e (resp. ouvert) de centre~0 et de polyrayon~$\br$. Puisque $y\in\Int(Y/Z)$, il existe un polyrayon~$\br$ et un diagramme commutatif de la forme
\[\xymatrix{
& Z \times \D^+_{\br} \ar[d]^\pi\\
Y \ar[r]^{\varphi} \ar[ru]^\psi & Z
},\]
o\`u~$\psi$ est une immersion ferm\'ee qui envoie~$y$ dans $Z\times \D^-_{\br}$. En passant aux r\'eductions, nous obtenons le diagramme
\[\xymatrix{
& \ti Z \times \wti{\D^+_{\br}} \ar[d]^{\ti\pi}\\
\ti Y \ar[r]^{\ti\varphi} \ar[ru]^{\ti\psi} & \ti Z
}.\]
La $\ti k$-vari\'et\'e gradu\'ee $\wti{\D^+_{\br}}$ est l'analogue d'une droite et le morphisme~$\ti\pi$ poss\`ede une section nulle~$\sigma$. La condition $\psi(y) \in Z\times \D^-_{\br}$ se traduit alors par $\ti\psi(\ti y) = \sigma(\ti\varphi(\ti y))$.

En outre, d'apr\`es~\cite[proposition~3.1~(iii)]{TemkinII}, le morphisme~$\ti\psi$ est fini. On en d\'eduit que son image est ferm\'ee. Ce r\'esultat se d\'emontre comme dans le cadre non gradu\'e par l'analogue du th\'eor\`eme \og going-up \fg{} de I.~S.~Cohen et A.~Seidenberg. Par cons\'equent, $\sigma^{-1}(\ti\psi(\ti Y))$ est une partie ferm\'ee de~$\ti Z$ contenue dans $\ti\pi(\ti\psi(\ti Y)) = \ti\varphi(\ti Y) = \red(\varphi(Y))$ et contenant~$\ti\varphi(\ti y)$. Le r\'esultat s'ensuit.
\end{proof}

\begin{rem}
On pourrait \'enoncer et d\'emontrer ce r\'esultat dans le cadre strictement affino\"{\i}de, mais la d\'efinition de l'int\'erieur, parce qu'elle autorise des rayons arbitraires, rend l'op\'eration malais\'ee. Signalons cependant qu'A.~Ducros a proc\'ed\'e ainsi dans la preuve de~\cite[lemme~3.1.2]{salg} (en supposant le corps de base alg\'ebriquement clos).
\end{rem}

\begin{lem}\label{lem:ducros}
Supposons que la valuation du corps~$k$ n'est pas triviale et que l'espace~$X$ est strictement $k$-affino\"{\i}de et \'equi\-di\-men\-sion\-nel. Soit~$\ti{Z}$ un ferm\'e de Zariski connexe de~$\ti{X}$. Alors son tube $\red^{-1}(\ti{Z})$ est un ouvert connexe de~$X$.
\end{lem}
\begin{proof}
\`A l'aide d'une r\'ecurrence sur la dimension, on se ram\`ene au cas o\`u~$\ti Z$ est irr\'eductible. Notons~$\ti z$ son point g\'en\'erique. Choisissons un point~$z$ de~$X$ au-dessus de~$Z$.

Soit~$C$ une composante connexe de~$\red^{-1}(\ti Z)$. Le morphisme d'inclusion $C \to X$ est sans bord et le lemme~\ref{lem:reductioninterieur} montre alors que~$\red{C}$ est un ferm\'e de Zariski de~$\ti X$. En particulier, il contient un point ferm\'e~$\ti x$ de~$\ti X$.

Soit~$D$ la composante connexe de~$\red^{-1}(\ti Z)$ contenant~$z$. Le m\^eme argument que pr\'ec\'edemment montre que~$\red{D}$ est un ferm\'e de Zariski de~$\ti X$. Puisqu'il contient le point g\'en\'erique~$\ti z$, il est n\'ecessairement \'egal \`a~$\ti Z$ et contient donc le point~$\ti x$. Le lemme~\ref{lem:staff} montre alors que $C=D$, autrement dit, que~$\red^{-1}(\ti Z)$ est connexe.
\end{proof}

Rappelons qu'un polyrayon $\br=(r_{1},\dotsc,r_{n}) \in (\R_{+}^\ast)^n$ est dit $k$-libre si la famille $(r_{1},\dotsc,r_{n})$ est libre dans le $\Q$-espace vectoriel $\R_{+}^\ast/\sqrt{|k^\ast|}$ (\cf~\cite[1.1]{variationdimension}). Pour un tel polyrayon~$\br$, on note~$k_{\br}$ l'ensemble des s\'eries de la forme
\[f = \sum_{i_{1},\dotsc,i_{n}\ge 0} a_{i_{1},\dotsc,i_{n}}\, u_{1}^{i_{1}} \dotsb u_{n}^{i_{n}}\] 
telle que la famille $\big(a_{i_{1},\dotsc,i_{n}}\, r_{1}^{i_{1}} \dotsb r_{n}^{i_{n}}\big)_{i_{1},\dotsc,i_{n}\ge 0}$ soit sommable. Muni de la norme d\'efinie par 
\[\|f\|_{\br} = \sum_{i_{1},\dotsc,i_{n}\ge 0} a_{i_{1},\dotsc,i_{n}}\, r_{1}^{i_{1}} \dotsb r_{n}^{i_{n}},\]
c'est un corps valu\'e complet. Sa r\'eduction gradu\'ee se calcule simplement~:
\[\wti{k_{\br}} = \ti{k}[r_{1}^{-1}\,T_{1}, r_{1}\,T_{1}^{-1},\dotsc,r_{n}^{-1}\,T_{n}, r_{n}\,T_{n}^{-1}].\] 

\begin{thm}\label{thm:Bosch}
Supposons que l'espace~$X$ est \'equi\-di\-men\-sion\-nel. Soit~$\ti{Z}$ un ferm\'e de Zariski connexe de~$\ti{X}$. Alors son tube $\red^{-1}(\ti{Z})$ est un ouvert connexe de~$X$.
\end{thm}
\begin{proof}
Soit~$\br$ un polyrayon $k$-libre tel que l'espace $X\hat{\otimes}_{k}k_{\br}$ soit strictement $k_{\br}$-affino\"ide. Consid\'erons le diagramme commutatif
\[\xymatrix{
X \ho_{k} k_{\br} \ar[r]^\pi \ar[d]^\red & X \ar[d]^\red\\
\wti{X\ho_{k} k_{\br}} \ar[r]^{\ti\pi} & \ti X
}.\]

D'apr\`es~\cite[proposition~3.1~(i)]{TemkinII}, le morphisme~$\ti\pi$ n'est autre que le morphisme de changement de base \`a~$\wti{k_{\br}}$. La formule explicite indiqu\'ee plus haut montre que ce morphisme pr\'eserve la connexit\'e. Par cons\'equent, $\ti\pi^{-1}(\ti Z)$ est un ferm\'e de Zariski connexe de $\wti{X\ho_{k} k_{\br}}$. D'apr\`es le lemme~\ref{lem:ducros}, son image r\'eciproque par~$\red^{-1}$ est connexe et l'on conclut gr\^ace \`a la surjectivit\'e du morphisme~$\pi$.
\end{proof}

\begin{rem}
Nous ignorons si le r\'esultat vaut sans l'hypoth\`ese d'\'equidimensionalit\'e.
\end{rem}

\section{Lieu $H$-strict d'un espace affino\"{\i}de}

Dans cette partie, nous fixons un espace $k$-affino\"{\i}de $X=\Mc(\As)$ et un sous-groupe~$H$ de~$\R_{+}^\ast$ contenant~$|k^\ast|$. Nous notons, comme d'habitude, $\rho$ la semi-norme spectrale sur~$\As$.

Nous allons effectuer quelques rappels sur la notion d'espace analytique $H$-strict, qui g\'en\'eralise celle d'espace strictement $k$-analytique (que l'on retrouve en choisissant $H=|k^\ast|$). Pour un traitement complet, nous renvoyons \`a~\cite[section~7]{CTdescent}.

\begin{defi}
L'espace $k$-affino\"{\i}de $X=\Mc(\As)$ est dit $H$-strict si, pour tout \'el\'ement~$f$ de~$\As$, on a $\rho(f)\in \sqrt{H}\cup\{0\}$.
\end{defi}

Cette d\'efinition s'\'etend \`a des espaces $k$-analytiques g\'en\'eraux en requ\'erant l'existence d'une structure $H$-stricte, c'est-\`a-dire d'un maillage par des domaines analytiques $H$-stricts. 

\'Enon\c{c}ons maintenant un analogue local de la d\'efinition valant pour les germes d'espaces en un point. Pour ce faire, nous supposerons que le groupe~$H$ n'est pas trivial. Cette condition est n\'ecessaire pour qu'un point admettant un voisinage affino\"{\i}de $H$-strict poss\`ede automatiquement un syst\`eme fondamental de tels voisinages. 

\begin{defi}
Supposons que $H \ne \{1\}$. Soit~$x$ un point de~$X$. Le germe~$(X,x)$ est dit $H$-strict si le point~$x$ poss\`ede un voisinage affino\"{\i}de $H$-strict.

On appelle lieu $H$-strict de~$X$ la partie~$X_{H}$ de~$X$ form\'ee des points~$x$ tels que le germe~$(X,x)$ est $H$-strict.
\end{defi}

De nouveau, cette d\'efinition s'\'etend \`a des espaces $k$-analytiques g\'en\'eraux en requ\'erant l'existence d'un voisinage $H$-strict (non n\'ecessairement affino\"{\i}de).

Donnons un exemple simple. \`A cet effet, introduisons la relation~$\prec$ sur~$\R_{+}$ d\'efinie par
\[a \prec b \textrm{ si } a<b \textrm{ ou } a=b=0.\]

\begin{lem}\label{lem:exemplelieuHstrict}
Supposons que $H \ne \{1\}$. Soient $f_{1},\dotsc,f_{a}, f'_{1},\dotsc,f'_{a} \in \As$ et $r_{1},\dotsc,r_{a} \in \R_{+}$. Soient $g_{1},\dotsc,g_{b},g'_{1},\dotsc,g'_{b}\in\As$ et $s_{1},\dotsc,s_{b} \in \sqrt{H}\cup \{0\}$. Soient $h_{1},\dotsc,h_{c},h'_{1},\dotsc,h'_{c})\in\As$ et $t_{1},\dotsc,t_{c}\in \R_{+}^\ast\setminus \sqrt{H}$. Alors, on a
\begin{align*}
&\bigl(\{x\in X \mid \forall i, |f_{i}(x)|< r_{i}\,|f'_{i}(x)|,\ \forall j, |g_{j}(x)|\le s_{j}\,|g'_{j}(x)|,\ \forall k, |h_{k}(x)|\le t_{k}\,|h'_{k}(x)|\}\bigr)_{H}\\ 
=& \{x\in X \mid \forall i, |f_{i}(x)|< r_{i}\,|f'_{i}(x)|,\ \forall j, |g_{j}(x)|\le s_{j}\,|g'_{j}(x)|,\ \forall k, |h_{k}(x)|\prec t_{k}\,|h'_{k}(x)|\}.
\end{align*}
\end{lem}
\begin{proof}
Posons
\[Y = \{x\in X \mid \forall i, |f_{i}(x)|< r_{i}\,|f'_{i}(x)|,\ \forall j, |g_{j}(x)|\le s_{j}\,|g'_{j}(x)|,\ \forall k, |h_{k}(x)|\le t_{k}\,|h'_{k}(x)|\}.\]
Il est imm\'ediat que l'espace apparaissant au membre de droite de l'\'egalit\'e appartient au lieu $H$-strict de~$Y$. Il suffit donc de montrer qu'aucun point~$x$ de~$Y$ pour lequel il existe un~$k$ tel que $|h_{k}(x)| = t_{k}\, |h'_{k}(x)|$ et $h'_{k}(x)\ne 0$ n'y appartient. 

Consid\'erons donc un tel~$x$. Pour tout voisinage affino\"{\i}de~$V$ de~$x$ dans~$Y$ assez petit, la fonction~$h'_{k}$ est inversible sur~$V$ et la norme spectrale de~$h_{k}/h'_{k}$ sur~$V$ est \'egale \`a~$t_{k}$, interdisant ainsi au voisinage d'\^etre $H$-strict. 
\end{proof}

Le fait qu'un germe soit $H$-strict peut se traduire purement en termes de la r\'eduction gradu\'ee de ce germe, telle qu'elle a \'et\'e d\'efinie par M.~Temkin dans~\cite[section~4]{TemkinII}. Rappelons que dans le cas o\`u~$x$ est un point d'un espace $k$-affino\"{\i}de $X=\Mc(\As)$, il s'agit de l'espace \og birationnel \fg{} suivant~:
\[\ti{X}_{x} = \P_{\wti{\Hs(x)}/\ti{k}}\{\ti{\chi}_{x}(\ti{\As})\},\]
c'est-\`a-dire de l'ensemble des anneaux de valuation gradu\'ee de corps des fractions gradu\'e $\wti{\Hs(x)}$ qui contiennent~$\ti{\chi}_{x}(\ti{\As})$ et~$\ti{k}$.

M.~Temkin a \'egalement d\'efini la propri\'et\'e d'\^etre $H$-strict pour de tels espaces (\cf~\cite[paragraphe pr\'ec\'edant la proposition~2.5]{TemkinII} et~\cite[fin de la section~5]{CTdescent}). Nous nous contenterons de mentionner la caract\'erisation suivante dans le cas qui nous int\'eresse ici.

\begin{prop}[M.~Temkin, \protect{\cite[proposition~2.5]{TemkinII}}]
Soit~$x\in X$. L'espace $\P_{\wti{\Hs(x)}/\ti{k}}\{\ti{\chi}_{x}(\ti{\As})\}$ est $H$-strict si, et seulement si, $\rho(\ti{\chi}_{x}(\ti{\As})) \subset \sqrt{H}\cup\{0\}$.
\end{prop}

Les propri\'et\'es pour les germes et leur r\'eduction co\"{\i}ncident. Ce r\'esultat vaut d'ailleurs pour un espace $k$-analytique~$X$ arbitraire.

\begin{thm}[B.~Conrad--M.~Temkin, \protect{\cite[theorem~7.5]{CTdescent}}]
Supposons que $H \ne \{1\}$. Pour tout point~$x$ de~$X$, le germe~$(X,x)$ est $H$-strict si, et seulement si, sa r\'eduction gradu\'ee~$\ti{X}_{x}$ est $H$-stricte.
\end{thm}

En combinant ces deux r\'esultats, nous obtenons la caract\'erisation suivante.

\begin{cor}\label{cor:IH}
Supposons que $H \ne \{1\}$. Notons~$\ti{I}_{H}$ l'id\'eal de~$\ti\As$ engendr\'e par l'ensemble des \'el\'ements homog\`enes dont l'ordre n'appartient pas \`a $\sqrt{H}\cup\{0\}$ et~$\ti{Z}_{H}$ le ferm\'e de Zariski de $\Spec(\ti\As)$ qu'il d\'efinit. Alors le lieu $H$-strict de~$X$ n'est autre que le tube de~$\ti{Z}_{H}$~: $X_{H} = \red^{-1}(\ti{Z}_{H})$.
\end{cor}
\begin{proof}
Soit $x\in X$. Ce point appartient \`a~$X_{H}$ si, et seulement si, on a $\rho(\ti{\chi}_{x}(\ti{\As})) \subset \sqrt{H}\cup\{0\}$, autrement dit, si, et seulement si, tout \'el\'ement de~$\ti{\As}$ dont l'ordre n'appartient pas \`a~$\sqrt{H}$ est envoy\'e sur~0 par~$\ti{\chi}_{x}$. Cette derni\`ere condition \'equivaut encore \`a l'inclusion $\ti{I}_{H} \subset \Ker(\ti{\chi}_{x})$ et le r\'esultat s'ensuit. 
 \end{proof}

\section{Rationalit\'e des sauts}

Soit~$X=\Mc(\As)$ un espace $k$-affino\"{\i}de. Dans cette derni\`ere section, nous d\'emontrons que les \'el\'ements $a_{0},\dotsc,a_{r}$ qui apparaissent dans l'\'enonc\'e du corollaire~\ref{cor:partitionsansrhoAs} appartiennent \`a~$\sqrt{\rho(\As)}$, concluant ainsi la preuve du th\'eor\`eme~\ref{thm:partition}.

\begin{lem}\label{lem:Hstrictconnexe}
Soit~$H$ un sous-groupe non trivial de~$\R_{+}^\ast$ contenant~$|k^\ast|$. Soit~$V$ un domaine affino\"{\i}de connexe de~$X$ d'alg\`ebre~$\Bs$ tel que
\[1 \notin \{uv \mid u,v\in\rho(\Bs), v \notin \sqrt{H}\}.\]
Alors, le lieu $H$-strict de~$V$ est connexe.
\end{lem}
\begin{proof}
Notons~$\ti{I}_{H}$ l'id\'eal de~$\wti\Bs$ engendr\'e par l'ensemble des \'el\'ements homog\`enes dont l'ordre n'appartient pas \`a~$\sqrt{H}\cup\{0\}$ et~$\ti{Z}_{H}$ le ferm\'e de Zariski de $\Spec(\wti\Bs)$ qu'il d\'efinit. D'apr\`es le corollaire~\ref{cor:IH}, le lieu $H$-strict de~$V$ n'est autre que le tube de~$\ti{Z}_{H}$.

Soit~$\ti{f}$ un idempotent homog\`ene non nul de~$\wti\Bs/\ti{I}_{H}$. De l'\'egalit\'e $\tilde{f}^2 = \tilde{f}$, on d\'eduit $\rho(\ti{f})=1$. Par cons\'equent, l'\'el\'ement~$\ti{f}$ se rel\`eve en un \'el\'ement homog\`ene~$\ti{F}$ de~$\wti\Bs$ tel que $\rho(\ti{F})=1$ et $\ti{F}^2-\ti{F} \in \ti{I}_{H}$. Par hypoth\`ese, l'id\'eal~$\ti{I}_{H}$ ne contient aucun \'el\'ement d'ordre~1. On en d\'eduit que~$\ti{F}$ est idempotent. Puisque~$V$ est connexe, sa r\'eduction l'est aussi et~$\ti{F}$ ne peut \^etre qu'un idempotent trivial. Il en est donc de m\^eme pour~$\ti{f}$.

Nous venons de montrer que~$\ti{Z}_{H}$ est un ferm\'e de Zariski connexe de~$\ti{V}$. D'apr\`es le th\'eor\`eme~\ref{thm:Bosch}, son tube est encore connexe.
\end{proof}

Pour toute partie~$P$ de~$\R_{+}^\ast$, nous noterons $\la P \ra_{\mon}$ (resp. $\la P \ra$) le sous-mono\"{\i}de (resp. sous-groupe) de~$\R_{+}^\ast$ qu'elle engendre.  Pour tout sous-mono\"{\i}de~$M$ de~$\R_{+}^\ast$, nous noterons~$\sqrt{M}$ le sous-mono\"{\i}de divisible de~$\R_{+}^\ast$ qu'il engendre.

\begin{lem}\label{lem:rhofr}
Soient $f_{1},\dotsc,f_{n}, g \in\As$ tels que $(f_{1},\dotsc,f_{n},g)=(1)$ et $r_{1},\dotsc,r_{n} \in\R_{+}^\ast$. Nous avons 
\[\rho(\{x\in X \mid \forall i, |f_{i}(x)| \le r_{i}\, |g(x)|\}) \subset \sqrt{\la \rho(\As),r_{1},\dotsc,r_{n}\ra_{\mon}}.\]
\end{lem}
\begin{proof}
La partie~$V$ de~$X$ d\'efinie par la conjonction des in\'egalit\'es $|f_{i}|\le r_{i}\, |g|$ est un domaine affino\"{\i}de d'alg\`ebre
\[\As_{V} = \As\{r_{1}^{-1}\, T_{1},\dotsc, r_{n}^{-1}\, T_{n}\}/(f_{1}-T_{1}\,g,\dotsc, f_{n}-T_{n}\,g).\]
Par d\'efinition de la norme sur $\Bs = \As\{r_{1}^{-1}\, T_{1},\dotsc, r_{n}^{-1}\, T_{n}\}$, nous avons $\rho(\Bs) = \la \rho(\As), r_{1},\dotsc,r_{n}\ra_{\mon}$. D'apr\`es~\cite[proposition~3.1~(iii)]{TemkinII}, le morphisme $\ti{\Bs} \to \ti{\As}_{V}$ induit par le morphisme quotient est fini. On en d\'eduit que $\rho(\As_{V}) \subset \sqrt{\la \rho(\As), r_{1},\dotsc,r_{n}\ra_{\mon}}$.
\end{proof}

\begin{cor}
Soit~$H$ un sous-groupe de~$\R_{+}^\ast$ contenant~$|k^\ast|$. Soit~$Y$ un espace $k$-analytique $H$-strict. Soient $f_{1},\dotsc,f_{n}, g \in\Os(Y)$ sans z\'eros communs et $r_{1},\dotsc,r_{n} \in\R_{+}^\ast$. Alors l'espace $k$-analytique 
\[\{x\in X \mid \forall i, |f_{i}(x)| \le r_{i}\, |g(x)|\}\]
est $\la H,r_{1},\dotsc,r_{n}\ra$-strict.
\end{cor}

Rappelons que nous avons d\'efini une relation~$\prec$ sur~$\R_{+}$ par
\[a \prec b \textrm{ si } a<b \textrm{ ou } a=b=0.\]

\begin{lem}\label{lem:adherence}
Soit~$H$ un sous-groupe de~$\R_{+}^\ast$ contenant~$|k^\ast|$. Soit~$Y$ un espace $k$-analytique $H$-strict. Soient $f, g \in\Os(Y)$ et $r \in\R_{+}^\ast \setminus \sqrt{H}$. Pour toute composante irr\'eductible~$F$ de $\{y\in Y \mid |f(y)| \le r\, |g(y)|\}$, nous avons soit
\[F \subset \{y\in Y \mid f(y)=g(y)=0\},\]
soit
\[\overline{F \cap \{y\in Y \mid |f(y)| < r\, |g(y)|\}} =F,\]
o\`u le surlignement d\'esigne l'adh\'erence topologique.

Dans tous les cas, $F \cap \{y\in Y \mid |f(y)| \prec r\, |g(y)|\}$ est dense dans~$F$. En particulier, pour toute composante connexe~$C$ de $\{y\in Y \mid |f(y)| \le r\, |g(y)|\}$, $C \cap \{y\in Y \mid |f(y)| \prec r\, |g(y)|\}$ est dense dans~$C$.
\end{lem}
\begin{proof}
Commen\c{c}ons par traiter le cas o\`u~$H$ n'est pas trivial. Soit~$F$ une composante irr\'eductible de $\{y\in Y \mid |f(y)| \le r\, |g(y)|\}$. Supposons qu'il existe un point~$y$ de~$F$ qui ne soit pas dans l'adh\'erence de $G = F \cap \{y\in Y \mid |f(y)| < r\, |g(y)|\}$. Il existe alors un voisinage~$U$ du point~$y$ dans~$F$ contenu dans $\{y\in Y \mid |f(y)| = r\, |g(y)|\}$. Nous pouvons supposer que~$U$ est ouvert. 

Soit~$z$ un point de~$U$ en lequel~$g$ ne s'annule pas. Ce point poss\`ede alors un voisinage~$V$ dans~$Y$ sur lequel la fonction~$f$ est inversible et la norme spectrale de~$g/f$ \'egale \`a~$r^{-1}$. Puisque $|g(z)|/|f(z)| = r^{-1}$, ces propri\'et\'es valent pour tout voisinage de~$z$ assez petit et nous pouvons donc supposer que~$V$ est compact et $H$-strict. En recouvrant~$V$ par un nombre fini de domaines affino\"{\i}des $H$-stricts, on aboutit \`a une contradiction.

Nous venons de montrer que la fonction~$g$ s'annule sur~$U$. On en d\'eduit qu'elle est nulle sur~$F$.

Supposons, maintenant, que $H=\{1\}$. On choisit alors un \'el\'ement $s \in \R_{+}^\ast \setminus r^\Z$ et on se ram\`ene au cas pr\'ec\'edent \`a l'aide du morphisme de projection $\pi \colon X \ho_{k} k_{s} \to X$ (voir les rappels pr\'ec\'edant le th\'eor\`eme~\ref{thm:Bosch} pour la notation~$k_{s}$). En effet, l'espace $X \ho_{k} k_{s}$ est alors $\la H, s \ra$-strict et $r\notin \sqrt{\la H, s\ra}$.
\end{proof}

\begin{prop}\label{prop:invarianceaffinoide}
Soient $f, g \in\As$ tels que $(f,g)=(1)$ et $r \in\R_{+}^\ast \setminus \sqrt{\rho(\As)}$. L'application naturelle
\[\pi_{0}(\{x\in X \mid |f(x)| < r\, |g(x)|\}) \to \pi_{0}(\{x\in X \mid |f(x)|\le r\, |g(x)|\})\]
est bijective.

Le m\^eme r\'esultat vaut en rempla\c{c}ant les composantes connexes par les composantes irr\'eductibles.
\end{prop}
\begin{proof}
Posons $V = \{x\in X \mid |f(x)|\le r\,|g(x)|\}$. Nous pouvons supposer que cette partie n'est pas vide. Consid\'erons l'une de ses composantes connexes~$C$. Il suffit de montrer que $C \cap \{x\in X \mid |f(x)| < r\, |g(x)|\}$ est connexe et non vide. La derni\`ere propri\'et\'e d\'ecoule du lemme~\ref{lem:adherence} en remarquant que~$f$ et~$g$ n'ont aucun z\'ero commun.

Supposons tout d'abord que $|k^\ast|\ne \{1\}$. Dans ce cas, d'apr\`es le lemme~\ref{lem:injectif}, nous pouvons trouver un \'element $r'>r$ avec $r'\in \sqrt{|k^\ast|}$ et une composante connexe~$D$ de $\{x\in X \mid |f(x)| \le r'\,|g(x)|\}$ tels que $C = D \cap \{x\in X \mid |f(x)| \le r\,|g(x)|\}$. D'apr\`es le lemme~\ref{lem:rhofr}, nous avons $\rho(\{x\in X \mid |f(x)| \le r'\,|g(x)|\}) \subset \sqrt{\rho(\As)}$. Choisissons une fonction~$g$ sur $\{x\in X \mid |f(x)| \le r'\,|g(x)|\}$ qui vaut~1 sur~$D$ et~0 sur les autres composantes connexes. Nous avons alors $D = \{x\in D \mid |g(x)| \ge 1\}$ et une nouvelle utilisation du lemme~\ref{lem:rhofr} montre que~$\rho(D) \subset \sqrt{\rho(\As)}$. 

Quitte \`a remplacer~$X$ par~$D$, nous pouvons donc supposer que~$X$ est connexe. D'apr\`es le lemme~\ref{lem:rhofr}, nous avons $\rho(V) \subset \sqrt{\la \rho(\As), r \ra_{\mon}}$. La condition du lemme~\ref{lem:Hstrictconnexe} est donc v\'erifi\'ee avec $H=\sqrt{\rho(\As)}$ et nous pouvons alors conclure en utilisant le lemme~\ref{lem:exemplelieuHstrict} (et, de nouveau, le fait que~$f$ et~$g$ n'ont aucun z\'ero commun).

Si $|k^\ast|=\{1\}$, on choisit un \'el\'ement $s \in \R_{+}^\ast \setminus \sqrt{\la \rho(\As), r\ra}$ et on se ram\`ene au cas pr\'ec\'edent \`a l'aide du morphisme de projection $\pi \colon X \ho_{k} k_{s} \to X$. Un raisonnement analogue \`a celui de la d\'emonstration du th\'eor\`eme~\ref{thm:Bosch} montre en effet que cette application pr\'eserve la connexit\'e par image directe et r\'eciproque.

La derni\`ere partie de l'\'enonc\'e se d\'emontre en appliquant le r\'esultat \`a la normalis\'ee de l'espace~$X$, op\'eration qui ne modifie pas~$\sqrt{\rho(\As)}$.
\end{proof}

\begin{cor}\label{cor:precaffinoide}
Soient $f, g \in\As$ et $r \in\R_{+}^\ast \setminus \sqrt{\rho(\As)}$. L'application naturelle
\[\pi_{0}(\{x\in X \mid |f(x)| \prec r\, |g(x)|\}) \to \pi_{0}(\{x\in X \mid |f(x)|\le r\, |g(x)|\})\]
est bijective.
\end{cor}
\begin{proof}
La preuve commence de la m\^eme fa\c{c}on que la pr\'ec\'edente. Nous pouvons supposer que $V = \{x\in X \mid |f(x)|\le r\,|g(x)|\}$ n'est pas vide. Consid\'erons l'une de ses composantes connexes~$C$. D'apr\`es le lemme~\ref{lem:adherence}, la partie $C \cap \{x\in X \mid |f(x)| \prec r\, |g(x)|\}$ n'est pas vide et il suffit de montrer qu'elle est connexe.

Le m\^eme argument que celui utilis\'e \`a la fin de la preuve pr\'ec\'edente permet de supposer que~$k$ n'est pas trivialement valu\'e. 

Notons~$F_{1},\dotsc,F_{m}$ les composantes irr\'eductibles de~$C$. Soit $i\in\cn{1}{m}$. Si $F_{i} \subset \{x\in X \mid f(x)=g(x)=0\}$, alors $\{x\in F_{i} \mid |f(x)|\prec r\, |g(x)|\} = F_{i}$ est connexe.

Sinon, d'apr\`es le corollaire~\ref{cor:ggeepsirreductible}, pour tout~$\eps>0$ assez petit, l'espace $F_{i,\eps} = \{x\in F_{i} \mid |g(x)| \ge \eps\}$ est connexe. Si l'on suppose de plus que $\eps\in \sqrt{|k^\ast|}$, alors, d'apr\`es le lemme~\ref{lem:rhofr}, en notant~$\Bs_{i,\eps}$ l'alg\`ebre de~$F_{i,\eps}$, nous avons $\sqrt{\rho(\Bs_{i,\eps})} \subset \sqrt{\rho(\As)}$. Nous pouvons donc appliquer la proposition pr\'ec\'edente et en d\'eduire que $\{x\in F_{i,\eps} \mid |f(x)| < r|g(x)|\}$ est connexe. En \'ecrivant $\{x\in F_{i} \mid g(x) \ne 0\}$ comme la r\'eunion des espaces~$F_{i,\eps}$ avec $\eps\in \sqrt{|k^\ast|}$ assez petit, on montre que l'espace
\[\{x\in F_{i} \mid g(x) \ne 0 \textrm{ et } |f(x)| < r|g(x)|\} = \{x\in F_{i} \mid |f(x)| < r|g(x)|\}\]
est connexe. Or, d'apr\`es le lemme~\ref{lem:adherence}, celui-ci est dense dans~$F_{i}$ et l'on en d\'eduit que $\{x\in F_{i} \mid |f(x)| \prec r|g(x)|\}$ est connexe.

Nous venons donc de montrer que, pour toute composante irr\'eductible~$F_{i}$ de~$C$, $F_{i}\cap \{x\in X \mid |f(x)| \prec r\, |g(x)|\}$ est connexe. Pour montrer que $D = C \cap \{x\in X \mid |f(x)| \prec r\, |g(x)|\}$ est connexe et conclure, il suffit de montrer que si~$F_{i}$ et~$F_{j}$ sont deux composantes irr\'eductibles de~$C$ qui se coupent, alors il existe des points~$x_{i}$ de~$F_{i}$ et~$x_{j}$ de~$F_{j}$ qui appartiennent \`a la m\^eme composante connexe de~$D$.

Consid\'erons donc deux telles composantes~$F_{i}$ et~$F_{j}$. Si la fonction~$g$ s'annule sur~$F_{i} \cap F_{j}$, alors 
\[\{x\in F_{i} \mid |f(x)| \prec r|g(x)|\} \cap \{x\in F_{j} \mid |f(x)| \prec r|g(x)|\} \ne \emptyset\]
et le r\'esultat est imm\'ediat.

Supposons, \`a pr\'esent, que la fonction~$g$ est minor\'ee en valeur absolue sur~$F_{i}\cap F_{j}$ par une constante~$\alpha>0$. Nous pouvons supposer que $\alpha\in\sqrt{|k^\ast|}$. Alors, d'apr\`es le lemme~\ref{lem:adherence}, il existe des points $x_{i}\in F_{i}$ et $x_{j}\in F_{j}$ v\'erifiant $|f(x_{i})| < r\, |g(x_{i})|$ et $|f(x_{j})| < r\, |g(x_{j})|$. En particulier, la fonction~$g$ ne s'annule en aucun de ces deux points et, quitte \`a diminuer~$\alpha$, nous pouvons supposer que $|g(x_{i})|\ge\alpha$ et $|g(x_{j})|\ge\alpha$.

Notons $L = \{\ell\in\cn{1}{m} \mid F_{\ell} \cap F_{i} \cap F_{j} \ne \emptyset\}$. D'apr\`es le corollaire~\ref{cor:ggeepsirreductible}, quitte \`a diminuer encore~$\alpha$, nous pouvons supposer que, pour tout~$\ell\in L$, $\{x\in F_{\ell} \mid |g(x)|\ge \alpha\}$ est connexe.  Dans ce cas, la partie
\[\bigcup_{\ell \in L} \{x\in F_{\ell} \mid |g(x)|\ge \alpha\}\] 
est une partie connexe de $\{x\in X \mid |g(x)|\ge \alpha \textrm{ et } |f(x)|\le r \,|g(x)|\}$. En utilisant de nouveau le lemme~\ref{lem:rhofr} et la proposition~\ref{prop:invarianceaffinoide}, on montre que la partie
\[\bigcup_{\ell \in L} \{x\in F_{\ell} \mid |g(x)|\ge \alpha\} \cap \{x\in X\mid |f(x)| < r\,|g(x)|\}\] 
est contenue dans une partie connexe de~$D$. Puisqu'elle contient les points~$x_{i}$ et~$x_{j}$, nous pouvons conclure.
\end{proof}

En combinant ce r\'esultat \`a celui du corollaire~\ref{cor:partitionsansrhoAs}, on obtient le th\'eor\`eme~\ref{thm:partition} annonc\'e en introduction.

\begin{cor}\label{cor:YHstrictfg}
Soit~$H$ un sous-groupe de~$\R_{+}^\ast$ contenant~$|k^\ast|$. Soit~$Y$ un espace $k$-analytique $H$-strict. Soient $f,g\in\Os(Y)$ et $s\in \R_{+}^\ast\setminus \sqrt{H}$. Les applications naturelles 
\[\pi_{0}(\{y\in Y \mid |f(y)| \prec s\, |g(y)|\}) \to \pi_{0}(\{y\in Y \mid |f(y)| \le s\, |g(y)|\})\]
sont bijectives.
\end{cor}
\begin{proof}
Nous pouvons supposer que $V = \{y\in Y \mid |f(y)| \le s\, |g(y)|\}$ n'est pas vide. Soit~$C$ une composante connexe de~$V$. Posons $V^\prec = \{y\in Y \mid |f(y)| \prec s\, |g(y)|\}$. Il suffit de montrer que $D = C \cap V^\prec$ est connexe et non vide. D'apr\`es le lemme~\ref{lem:adherence}, nous avons~$\overline{D} = C$. En particulier, $D$~n'est pas vide.

Notons $(D_{i})_{i\in I}$ la famille des composantes connexes de~$D$. Soit~$y$ un point de~$C\setminus D$. Puisque l'espace~$Y$ est $H$-strict, il existe des domaines affino\"{\i}des $H$-stricts $V_{1},\dotsc,V_{n}$ de~$Y$ contenant~$y$ dont la r\'eunion est un voisinage de~$y$ et dont les intersections deux \`a deux sont $H$-strictes. 

Pour tout $p\in\cn{1}{n}$, $C\cap V_{p}$ est compact et poss\`ede donc un nombre fini de composantes connexes. D'apr\`es le corollaire~\ref{cor:precaffinoide}, il en va de m\^eme pour $C\cap V_{p} \cap V^\prec$. En particulier, $V_{p}$ ne rencontre qu'un nombre fini de composantes~$D_{i}$. 

Puisque la r\'eunion des~$V_{p}$ forme un voisinage de~$y$, il existe donc un \'el\'ement~$i\in I$ tel que $y\in \overline{D_{i}}$. Nous venons de montrer que 
\[C = \overline{D} = \bigcup_{i\in I} \overline{D_{i}}.\]

Soient maintenant $i,j\in I$ tels que $y \in \overline{D}_{i} \cap \overline{D}_{j}$. Soit~$p$ tel que $y \in \overline{D_{i} \cap V_{p}}$ et soit~$q$ tel que $y \in \overline{D_{j} \cap V_{q}}$. Il existe un domaine affino\"{\i}de $H$-strict~$W$ de~$V_{p}\cap V_{q}$ contenant~$y$. D'apr\`es le lemme~\ref{lem:adherence} et le corollaire~\ref{cor:precaffinoide}, les espaces $W \cap V^\prec$, $V_{p}\cap V^\prec$ et $V_{q}\cap V^\prec$ ne sont pas vides et il existe une unique composante connexe~$E$ (resp.~$E_{p}$, resp.~$E_{q}$) de $W \cap V^\prec$ (resp. $V_{p}\cap V^\prec$, resp. $V_{q}\cap V^\prec$) telle que $y\in \overline{E}$ (resp. $y\in \overline{E}_{p}$, resp. $y\in \overline{E}_{q}$).

On en d\'eduit d'une part que~$E$ est contenu dans~$E_{p}$ et~$E_{q}$ et d'autre part que ces deux derniers espaces rencontrent, et donc sont contenus dans, $D_{i}$ et~$D_{j}$ respectivement. Par cons\'equent, $i=j$. 

Finalement, nous avons montr\'e que~$C$ est l'union des adh\'erences des composantes~$D_{i}$ et que ces adh\'erentes sont disjointes. Puisqu'elles sont connexes, ce sont les composantes connexes de~$C$. On en d\'eduit que~$I$ est r\'eduit \`a un \'el\'ement et donc que~$D$ est connexe.
\end{proof}

\`A l'aide d'une r\'ecurrence bas\'ee sur le lemme~\ref{lem:rhofr} et le corollaire~\ref{cor:YHstrictfg}, nous pouvons finalement obtenir le r\'esultat qui suit.

\begin{cor}
Soit~$H$ un sous-groupe de~$\R_{+}^\ast$ contenant~$|k^\ast|$. Soit~$Y$ un espace $k$-analytique $H$-strict. Soient $f_{1},\dotsc,f_{n},g_{1},\dotsc,g_{n}\in\Os(Y)$ telles que, pour tout~$i$, $f_{i}$ et~$g_{i}$ ne s'annulent pas simultan\'ement sur~$Y$. Soient $r_{1},\dotsc,r_{n} \in\R_{+}^\ast$ tels que, pour tout~$i$,  $r_{i}$ soit d'image non nulle dans $\R_{+}^\ast/\sqrt{\la H, r_{1},\dotsc, r_{i-1}\ra}$, 

Alors, l'application naturelle
\[\pi_{0}\Big(\bigcap_{1\le i\le n} \{y\in Y \mid |f_{i}(y)| < r_{i}\, |g_{i}(y)|\}\Big) \to \pi_{0}\Big(\bigcap_{1\le i\le n} \{y\in Y \mid |f_{i}(y)| \le r_{i}\, |g_{i}(y)|\}\Big)\]
est bijective.
\end{cor}

\begin{rem}
M\^eme si nous ne l'avons pas mentionn\'e explicitement, tous les r\'esultats de cette section restent valables en rempla\c{c}ant l'ensemble des composantes connexes par l'ensemble des composantes connexes g\'eom\'etriques, l'ensemble des composantes irr\'eductibles ou l'ensemble des composantes irr\'eductibles g\'eom\'etriques. Il suffit, pour le d\'emontrer, de les appliquer \`a l'espace obtenu en \'etendant les scalaires au compl\'et\'e d'une cl\^oture alg\'ebrique de~$k$ ou \`a la normalis\'ee de l'espace.
\end{rem}

\nocite{}
\bibliographystyle{smfalpha}
\bibliography{biblio}

\providecommand{\bysame}{\leavevmode ---\ }
\providecommand{\og}{``}
\providecommand{\fg}{''}
\providecommand{\smfandname}{\&}
\providecommand{\smfedsname}{\'eds.}
\providecommand{\smfedname}{\'ed.}
\providecommand{\smfmastersthesisname}{M\'emoire}
\providecommand{\smfphdthesisname}{Th\`ese}
\begin{thebibliography}{BGR84}

\bibitem[AS02]{Abbes}
{\scshape A.~Abbes {\normalfont \smfandname} T.~Saito} -- {\og Ramification of
  local fields with imperfect residue fields\fg}, \emph{Amer. J. Math.}
  \textbf{124} (2002), no.~5, p.~879--920.

\bibitem[Bar70]{Bart}
{\scshape W.~Bartenwerfer} -- {\og Einige {F}ortsetzungss\"atze in der
  {$p$}-adischen {A}nalysis\fg}, \emph{Math. Ann.} \textbf{185} (1970),
  p.~191--210.

\bibitem[Ber90]{rouge}
{\scshape V.~G. Berkovich} -- \emph{Spectral theory and analytic geometry over
  non-{A}rchimedean fields}, Mathematical Surveys and Monographs, vol.~33,
  American Mathematical Society, Providence, RI, 1990.

\bibitem[Ber93]{bleu}
\bysame , {\og \'{E}tale cohomology for non-{A}rchimedean analytic spaces\fg},
  \emph{Inst. Hautes \'Etudes Sci. Publ. Math.} (1993), no.~78, p.~5--161
  (1994).

\bibitem[Ber99]{smoothI}
\bysame , {\og Smooth {$p$}-adic analytic spaces are locally contractible\fg},
  \emph{Invent. Math.} \textbf{137} (1999), no.~1, p.~1--84.

\bibitem[BGR84]{BGR}
{\scshape S.~Bosch, U.~G{\"u}ntzer {\normalfont \smfandname} R.~Remmert} --
  \emph{Non-{A}rchimedean analysis}, Grundlehren der Mathematischen
  Wissenschaften, vol. 261, Springer-Verlag, Berlin, 1984, A systematic
  approach to rigid analytic geometry.

\bibitem[BLR95]{FIV}
{\scshape S.~Bosch, W.~L{\"u}tkebohmert {\normalfont \smfandname} M.~Raynaud}
  -- {\og Formal and rigid geometry. {IV}. {T}he reduced fibre theorem\fg},
  \emph{Invent. Math.} \textbf{119} (1995), no.~2, p.~361--398.

\bibitem[Bos77]{Bosch}
{\scshape S.~Bosch} -- {\og Eine bemerkenswerte {E}igenschaft der formellen
  {F}asern affinoider {R}\"aume\fg}, \emph{Math. Ann.} \textbf{229} (1977),
  no.~1, p.~25--45.

\bibitem[CT10]{CTdescent}
{\scshape B.~Conrad {\normalfont \smfandname} M.~Temkin} -- {\og Descent for
  non-archimedean analytic spaces\fg}, 2010,
  \url{http://math.stanford.edu/~conrad/papers/descentnew.pdf},
  \url{http://www.math.huji.ac.il/~temkin/papers/Descent.pdf}.

\bibitem[Duc03]{salg}
{\scshape A.~Ducros} -- {\og Parties semi-alg\'ebriques d'une vari\'et\'e
  alg\'ebrique {$p$}-adique\fg}, \emph{Manuscripta Math.} \textbf{111} (2003),
  no.~4, p.~513--528.

\bibitem[Duc07]{variationdimension}
\bysame , {\og Variation de la dimension relative en g\'eom\'etrie analytique
  {p}-adique\fg}, \emph{Compos. Math.} \textbf{143} (2007), no.~6,
  p.~1511--1532.

\bibitem[Duc12]{moderes}
\bysame , {\og Les espaces de {B}erkovich sont mod\'er\'es, d'apr\`es
  {E}.~{H}rushovksi et {F}.~{L}oeser\fg}, \emph{S\'em. Bourbaki} (2012),
  no.~Exp. 1056.

\bibitem[Epp73]{Epp}
{\scshape H.~P. Epp} -- {\og Eliminating wild ramification\fg}, \emph{Invent.
  Math.} \textbf{19} (1973), p.~235--249.

\bibitem[HL10]{HL}
{\scshape E.~Hrushovski {\normalfont \smfandname} F.~Loeser} -- {\og
  Non-archimedean tame topology and stably dominated types\fg}, arXiv, 2010,
  \url{http://arxiv.org/abs/1009.0252}.

\bibitem[L{\"u}t74]{Lu}
{\scshape W.~L{\"u}tkebohmert} -- {\og Der {S}atz von {R}emmert-{S}tein in der
  nichtarchimedischen {F}unktionentheorie\fg}, \emph{Math. Z.} \textbf{139}
  (1974), p.~69--84.

\bibitem[Poi08]{Compositio}
{\scshape J.~Poineau} -- {\og Un r\'esultat de connexit\'e pour les
  vari\'et\'es analytiques {$p$}-adiques: privil\`ege et noeth\'erianit\'e\fg},
  \emph{Compos. Math.} \textbf{144} (2008), no.~1, p.~107--133.

\bibitem[Tem04]{TemkinII}
{\scshape M.~Temkin} -- {\og On local properties of non-{A}rchimedean analytic
  spaces. {II}\fg}, \emph{Israel J. Math.} \textbf{140} (2004), p.~1--27.

\end{thebibliography}

\end{document}